\documentclass[12pt,reqno,a4paper]{amsart}

\usepackage{amsmath,amsfonts,amsthm,amssymb,amsxtra,enumerate,mathtools,mathabx}
\usepackage[normalem]{ulem}
\usepackage[utf8]{inputenc}
\usepackage[T1]{fontenc}
\usepackage{lmodern}
\usepackage{bbm}
\usepackage{mathrsfs}
\usepackage{enumerate}
\usepackage{comment}
\usepackage{placeins}
\usepackage{pgfplots}
\usepackage{mathabx}
\pgfplotsset{compat=newest}
\usepgfplotslibrary{fillbetween}
\usepackage{tikz}
\usepackage{csquotes}

\newtheorem{theorem}{Theorem}[section]
\newtheorem{proposition}[theorem]{Proposition}
\newtheorem{lemma}[theorem]{Lemma}

\theoremstyle{definition}

\theoremstyle{remark}

\newtheorem{remark}[theorem]{Remark}

\numberwithin{equation}{section}

\newcommand{\dd}{\, \mathrm{d}}

\renewcommand{\epsilon}{\varepsilon}

\newcommand{\id}{\mathbb{I}}

\newcommand{\Lc}{\mathcal{L}}

\newcommand{\N}{\mathbb{N}}
\newcommand{\norm}[2][]{{\left\|#2\right\|}} 
\renewcommand{\phi}{\varphi}
\newcommand{\R}{\mathbb{R}}

\newcommand{\Z}{\mathbb{Z}}

\DeclareMathOperator{\Tr}{Tr}

\newcommand{\abs}[2][]{{\left\vert#2\right\vert}}

\usepackage{hyperref}
\begin{document}
	
	\title[On the asymptotics of low--lying states in the Stark effect
	]
	{On the asymptotic number of low--lying states in the two--dimensional confined Stark effect}
    \author[Larry Read]{Larry Read} 
	\address[Larry Read]{Mathematisches Institut, Ludwig-Maximilans Universit\"at M\"unchen, Theresienstr. 39, 80333 M\"unchen, Germany}
	\email{read@math.lmu.de}

    \begin{abstract}
        We investigate the Stark operator restricted to a bounded domain $\Omega\subset\R^2$ with Dirichlet boundary conditions. In the semiclassical limit, a three-term asymptotic expansion for its individual eigenvalues has been established, with coefficients dependent on the curvature of $\Omega$. We analyse the accumulation of eigenvalues beneath the leading-order terms in these expansions, establishing Weyl-type asymptotics. Furthermore, we derive weak asymptotics for the density of the spectral projector onto these low-lying states.
\end{abstract}

\maketitle
\section{Introduction}
Consider the operator \begin{equation}\label{eqn:opLh} 
    \Lc_h=-h^2\Delta+x_1 \text{ in }L^2(\Omega) 
\end{equation} 
with Dirichlet boundary conditions, where $\Omega\subset \R^2$ is an open, bounded and connected region. Suppose that there is a unique point $X_0=(x_0,y_0)\in \partial\Omega$ that minimises the first coordinate, and that around $X_0$ the boundary is smooth with positive curvature at $X_0$. Then, as the semiclassical parameter $h$ tends to zero, the bound states of \eqref{eqn:opLh} cluster near the boundary at $X_0$, where the repulsive potential is smallest. The confined Stark effect is characterised by the splitting of the energy levels in this limit.

In \cite{cornean_two-dimensional_2022} Cornean, Krej\v{c}i\v{r}\'{i}k
, Pedersen, Raymond and Stockmeyer examined this splitting. They determined that for any fixed $k\geq 1$ the $k$th eigenvalue of \eqref{eqn:opLh} satisfies
\begin{equation}\label{eqn:eigencornean}
    \lambda_k(\mathcal{L}_h)=x_{0}+z_1 h^{2/3}+(2k-1)\sqrt{\frac{\kappa_0}{2}}h+\mathcal{O}_{h\rightarrow 0_+}(h^{4/3})
\end{equation}
where $-z_1\approx- 2.338$ is the first zero of the the Airy function and $\kappa_0>0$ is the curvature of the boundary of $\Omega$ at $X_0$. We note that a higher-dimensional analogue of this expansion has been found in \cite{fahs}.

The idea behind the expansion \eqref{eqn:eigencornean} is that as the bound states of \eqref{eqn:opLh} become concentrated near $X_0$, the curvature in the boundary acts as an effective harmonic oscillator in the tangential component, whilst the orthogonal component produces the Airy zero. Indeed, the approach taken in \cite{cornean_two-dimensional_2022} was to construct quasi-states using the eigenfunctions of the Airy operator and harmonic oscillator in tubular coordinates around $X_0$ on a suitable scale in $h$. 

In this work, we are concerned with the directly related question of how many eigenvalues accumulate below the different levels in \eqref{eqn:eigencornean}. To introduce this, first consider the general case of a Schr\"odinger operator $-h^2\Delta+V$ in $L^2(\omega)$ with Dirichlet conditions, on an open set $\omega\subset\R^d$ with suitably regular potential $V$. The well-known Weyl's law states that the counting function, 
\begin{equation*}
    N(-h^2\Delta+V,\Lambda)=\#\left\{k\in \N \colon \lambda_k(-h^2\Delta+V)<\Lambda\right\},
\end{equation*}
satisfies the asymptotics 
\begin{align}\label{eqn:weylscounting}
    \lim_{h\rightarrow 0_+}h^d N(-h^2\Delta+V,\Lambda)=L_{0,d}^{\mathrm{cl}}\int_{\omega}\left(\Lambda-V(x)\right)_{+}^{d/2}\dd x
\end{align} 
where $a_\pm=(\abs{a}\pm a)/2$ and  $L_{0,d}^{\mathrm{cl}}$ is an instance of 
\begin{equation}\label{eqn:semiclassicalconst}
        L_{\gamma,d}^{\mathrm{cl}}=\frac{\Gamma(\gamma+1)}{(4\pi)^{d/2}\Gamma(\gamma+1+d/2)}.
\end{equation}
We investigate this limit for the Stark operator \eqref{eqn:opLh}, where for the counting function we take $\Lambda$ as a function of $h$, choosing either 
\begin{equation*}
         x_0+\mu h^{2/3} \text{ or }\\
         x_0+z_1h^{2/3}+\mu h^{\alpha},
\end{equation*}
with $\mu\geq 0$ and $\alpha\in (2/3,1)$. In doing so, we count the number of low-lying eigenvalues, corresponding to the expansion \eqref{eqn:eigencornean}. We note that a similar regime was considered by Frank in \cite{frank_asymptotic_2007}, who determined the asymptotic number of edge states for a magnetic Laplacian with Neumann boundary conditions.

\subsection{Main results} 
    
    Before stating our main results, we establish the construction of tubular coordinates about $X_0$. For the latter, we follow \cite{cornean_two-dimensional_2022} and its notation as closely as possible. Without loss of generality, we will assume that $x_0=0$.
    
    Consider an arc-length parameterisation $\gamma(s)$ of $\partial\Omega$ in the vicinity of $X_0$, such that $\gamma(0)=X_0$. The outward normal at any point on this curve can be represented by $n(s)=(\cos(\theta(s)),\sin(\theta(s)))$ and the curvature by $\kappa(s)=\theta^\prime(s)$, with $\kappa_0\coloneqq \kappa(0)$. We can then fix $\delta>0$ to be sufficiently small so that the mapping $\tau\colon (-\delta,\delta)\times (0,\delta)\rightarrow \Omega$ defined by 
    \begin{equation}\label{eqn:tubular}
    \tau(s,t)=\gamma(s)-t n(s)
    \end{equation}
    establishes a diffeomorphism between the strip and its image in $\Omega$, where the determinant of its Jacobian is given by $1-\kappa(s)t$.

    Our first result reveals Weyl-type asymptotics for the $\gamma-$Riesz means of the Stark operator \eqref{eqn:opLh}, $\Tr\left(\Lc_h-\Lambda\right)_{-}^\gamma$, for which we identify $\gamma=0$ with the counting function. 
    \begin{theorem}\label{thm:maintrace}
        Let $\gamma \geq 0$, $\alpha\in (2/3,1)$ and $\mu\geq 0$, then
        \begin{align*}
            \lim_{h\rightarrow 0_+}h^{(1-2\gamma)/3}\Tr\left(\Lc_h-\mu h^{2/3}\right)_{-}^\gamma&=\frac{4\pi L_{\gamma,2}^{\mathrm{cl}}}{\sqrt{2\kappa_0}}\sum_{k=1}\left(\mu-z_k\right)_+^{\gamma+1},\text{ and }\\
            \lim_{h\rightarrow 0_+}h^{1-\alpha(1+\gamma)}\Tr\left(\Lc_h-z_1h^{2/3}-\mu h^{\alpha}\right)_{-}^\gamma &=\frac{4\pi L_{\gamma,2}^{\mathrm{cl}}}{\sqrt{2\kappa_0}}\mu^{\gamma+1}
        \end{align*}
        where $-z_k$ is the $k$th zero of the Airy function and $L^{\mathrm{cl}}_{\gamma,2}$ is given by \eqref{eqn:semiclassicalconst}.
    \end{theorem}
    The approach we employ involves the use of Dirichlet--Neumann bracketing and a rescaling of tubular coordinates around $X_0$. The former operation is carried out exclusively in the parallel coordinate, whilst in the orthogonal component we construct quasi-states, as in \cite{cornean_two-dimensional_2022}. These quasi-states approximate the eigenfunctions of the operator 
    \begin{equation*}
        -\frac{\mathrm{d}^2}{\mathrm{d}t^2}+t \text{ in } L^2(\R_+)
    \end{equation*}
    with Dirichlet boundary conditions, $\R_+=(0,\infty)$,  which arises from the Taylor series of $\tau$ near $X_0$. The eigenvalues of this operator are the absolute values of the zeros of the Airy function, $z_k$. We denote the corresponding normalised eigenfunctions by $a_k$, which are given by 
    \begin{equation}\label{eqn:airnorm}
         \mathrm{a}_k(t)=\frac{\mathrm{Ai}(t-z_k)}{\norm{\mathrm{Ai}(\cdot-z_k)}_{2}^2} \text{ for } \ k \geq 1,
     \end{equation}
    where $\mathrm{Ai}$ denotes the Airy function.
    
    Now, let $\rho_h(\cdot;\Lambda)$ denote the density of $(\mathcal{L}_h-\Lambda)^0_{-}$. Recall that for a trace class operator $T$ represented by its Schmidt decomposition $T=\sum_{k=1}s_k(\cdot,\phi_k)\phi_k$ the density is defined as $\rho_T=\sum_{k=1}s_k\abs{\phi_k}^2$. Then, in our second result we find weak-type asymptotics for $\rho_h$ in tubular coordinates rescaled in $\R^2_+\coloneqq\{(s,t)\in \R^2\colon t> 0\}$.
    \begin{theorem}\label{thm:maindensity}
        Let $\alpha\in (2/3,1)$ and $\mu\geq 0$, then as $h\rightarrow 0_+$
        \begin{align*}
            h^{4/3}\rho_h(\tau(h^{1/3}s,h^{2/3}t);\mu h^{2/3})&\rightharpoonup \frac{1}{\pi}\sum_{k=1}\left(\mu-\frac{\kappa_0}{2}s^2-z_k\right)_+^{1/2}\mathrm{a}_k(t)^2, \text{ and }\\h^{5/3-\alpha/2}\rho_h(\tau(h^{\alpha/2}s,h^{2/3}t);z_1 h^{2/3}+\mu h^\alpha)&\rightharpoonup \frac{1}{\pi}\left(\mu-\frac{\kappa_0}{2}s^2\right)_{+}^{1/2}\mathrm{a}_1(t)^2
        \end{align*}
        in the sense of distributions on $\R^2_+$, where $\tau$ is given by \eqref{eqn:tubular}.
    \end{theorem}
    
    The structure of the paper is as follows. In Section \ref{sec:asymp1} we start with an approach using Dirichlet--Neumann bracketing to obtain a semiclassical result. In Section \ref{sec:maintrace} we then proceed to prove a generalised version of Theorem \ref{thm:maintrace} with the addition of a potential, which we achieve by rescaling and splitting more accurately in the orthogonal component. Finally, in Section \ref{sec:proj} we make use of this generalisation and proceed to prove Theorem \ref{thm:maindensity}. 
    \section{Semiclassical approximation for the counting function} \label{sec:asymp1}
    Our aim in this section is to deduce asymptotics for the counting function by bracketing in the tubular coordinates defined by \eqref{eqn:tubular}. According to the semiclassical approximation, it is suggested that
    \begin{align*}
        N(\Lc_h,\mu h^{3/2})&\approx (2\pi)^{-2}\abs{\left\{ (\xi,x)\in \R^2\times \Omega\colon h^2\abs{\xi}^2+x_1<\mu h^{3/2}\right\}}.
    \intertext{Thus for sufficiently small $h$ only the region near $X_0$ becomes relevant in the estimate, due to the positive curvature of $\Omega$ at $X_0$. Taking $\tau=(\tau_1,\tau_2)$ as in \eqref{eqn:tubular}, we note from \cite{cornean_two-dimensional_2022} that the Taylor series expansion for $\tau_1$ about $(0,0)$ is given by
    \begin{equation}\label{eqn:taylorexp}
        \tau_1(s,t)=t+\frac{\kappa_0}{2}s^2+\mathcal{O}_{t\rightarrow 0, s\rightarrow 0}(|s|^3+|t s^2|).
    \end{equation}
    It follows that}
        N(\Lc_h,\mu h^{3/2})&\approx \frac{1}{4\pi h^2}\int_{0}^{\delta}\int_{-\delta}^{\delta}(\mu h^{2/3}-\tau_1(s,t))_+ (1-\kappa(s)t)\dd s\dd t\\
        &\approx \frac{1}{4\pi h^2}\int_{0}^{\delta}\int_{-\delta}^{\delta}\left(\mu h^{2/3}-t-\frac{\kappa_0}{2}s^2\right)_+ (1-\kappa(s)t)\dd s\dd t\\
        &=\frac{4}{15 \pi \sqrt{2\kappa_0}}\mu^{5/2} h^{-1/3}+\mathcal{O}_{h\rightarrow 0_+}(h^{1/3}).
    \end{align*}
    Though this appears to contradict the statement of Theorem \ref{thm:maintrace}, we find that it holds for large $\mu$. The consistency is evident from the known asymptotics for the Airy zeros, which satisfy \[z_k=\frac{1}{4} (3 \pi )^{2/3} (4 k-1)^{2/3}(1+o_{k\rightarrow \infty}(1)).\] Thus as $\mu$ unlocks the levels in the leading order term the number of eigenvalues becomes semiclassical.  
    
    We start by restricting $\mathcal{L}_h$ to some region about $X_0$, sufficiently proportional to $h$, with Neumann and Dirichlet conditions. We fix $\eta\in (0,1/15)$ and define the region $\mathcal{W}_h=(-h^{1/3-\eta},h^{1/3-\eta})\times(0,h^{2/3-\eta})$, then it results from the variational principle that
    \begin{align*}
        \Lc_h-\mu h^{2/3}\leq \left(-\Delta_{\tau(\mathcal{W}_h)}^D+x_1-\mu h^{2/3}\right)&\oplus \left(-\Delta_{\Omega\backslash\tau(\mathcal{W}_h)}^D+x_1-\mu h^{2/3}\right), \text{ and }\\
        \Lc_h-\mu h^{2/3}\geq \left(-\Delta_{\tau(\mathcal{W}_h)}^N+x_1-\mu h^{2/3}\right)&\oplus \left(-\Delta_{\Omega\backslash\tau(\mathcal{W}_h)}^N+x_1-\mu h^{2/3}\right)
    \end{align*}
    in the sense of quadratic forms, where $-\Delta^{D}_{\omega}$ and $-\Delta^N_{\omega}$ denote the Dirichlet or Neumann Laplacians on $\omega\subset \R^2$. (In the case of the latter, we have the option to maintain or substitute the existing Dirichlet boundary conditions without altering the above relationship). Consequently, noting that $\{x\in \Omega\colon x_1<\mu h^{2/3}\}\subset \tau(\mathcal{W}_h)$ for sufficiently small $h$, it follows that 
    \begin{equation}\label{eqn:bracketing}
        N\left(-\Delta_{\tau(\mathcal{W}_h)}^D+x_1,\mu h^{2/3}\right)\leq N\left(\Lc_h,\mu h^{2/3}\right)\leq N\left(-\Delta_{\tau(\mathcal{W}_h)}^N+x_1,\mu h^{2/3}\right).
    \end{equation}
    Each of these operators, which are restrictions of $\mathcal{L}_h$ to $\tau(\mathcal{W}_h)$, can be written in tubular coordinates as
    \begin{equation*}
        -h^2 m^{-1} \partial_s m^{-1}\partial_s-h^2 m^{-1}\partial_t m \partial_t +\tau_1(s,t) \text{ in }L^2(\mathcal{W}_h,m\dd s \dd t),
    \end{equation*}
    where $\tau_1$ is the first coordinate of $\tau=(\tau_1,\tau_2)$ and $m(s,t)=1-\kappa(s)t$ is the Jacobian of $\tau$. Their quadratic forms correspond to 
    \begin{equation*}
        q_{h}[\psi]=\iint_{\mathcal{W}_h}\left[h^2\left(m^{-2}|\partial_s \psi|^2+|\partial_t\psi|^2\right)+\tau_1(s,t)|\psi|^2\right] m\dd s\dd t,
    \end{equation*}
    considered for suitable classes of $\phi\in H^1(\mathcal{W}_h)$ depending on the boundary conditions. 
    
    At this point, we wish to use our shrinking domain to change the operator to one on a flat strip, without the curvature term. Given the Taylor expansion \eqref{eqn:taylorexp} and the fact that the boundary of $\Omega$ is smooth and $\kappa_0>0$, we can take $h$ to be sufficiently small so that for every $(s,t)\in \mathcal{W}_{h}$
    \begin{align*}
        \abs{\tau_1(s,t)-\left(t+\frac{\kappa_0}{2}s^2\right)}\lesssim h^{1-3\eta}\text{ and }
        -h^{2/3-\eta}\lesssim m(s,t)-1\leq 0,
    \end{align*}
    with implicit constants that are independent of $h$. Then we use the above estimates to approximate $q_h$ from above and below correspondingly. To simplify notation, where we have used implicit constants above, we bound them by $h^{-\eta}$ and assume that $h$ is sufficiently small. We find that for any $\psi\in H^1(\mathcal{W}_h)$
    \begin{align*}
        q_{h}[\psi]-h^{1-4\eta}\norm{\psi}_2^2&\leq \left(1- h^{2/3-2\eta}\right)^{-1}\iint_{\mathcal{W}_h}h^2\left(|\partial_s \psi|^2+|\partial_t\psi|^2\right)+\left(t+\frac{\kappa_0}{2}s^2\right)|\psi|^2\dd s\dd t ,\intertext{ and }
        q_{h}[\psi]+h^{1-4\eta}\norm{\psi}_2^2&\geq \left(1-h^{2/3-2\eta}\right)\iint_{\mathcal{W}_h}h^2\left(|\partial_s \psi|^2+|\partial_t\psi|^2\right)+\left(t+\frac{\kappa_0}{2}s^2\right)|\psi|^2\dd s\dd t.
    \end{align*}
    Thus we reduce to operators in $L^2(\mathcal{W}_h,\dd s\dd t)$ and using the estimates \eqref{eqn:bracketing} together with those for $q_h$ above, we see that 
    \begin{equation}\label{eqn:DNBrack}
        \begin{split}
        N(\Lc_h,\mu h^{2/3})&\geq  N\left(-h^2\Delta_{\mathcal{W}_h}^D+\frac{\kappa_0}{2}s^2+t,h^{2/3}(\mu-h^{1/3-5\eta})\right) \text{ and }\\
      N(\Lc_h,\mu h^{2/3})&\leq N\left(-h^2\Delta_{\mathcal{W}_h}^N+\frac{\kappa_0}{2}s^2+t,h^{2/3}(\mu+h^{1/3-5\eta})\right).
      \end{split}
    \end{equation}
    In the following section, we will represent the errors in the preceding expression in a different manner. However, it will be crucial for our subsequent discussions that we are able to express it as $h^{1/3-5\eta}$.

    \begin{proposition}\label{prop:roughweyl}
        The following holds
        \begin{equation*}
            \lim_{h\rightarrow 0_+} h^{1/3}N(\Lc_h,\mu h^{2/3})=\frac{4\mu^{5/2}}{15\pi \sqrt{2\kappa_0}}\left(1+\mathcal{O}_{\mu\rightarrow \infty}(\mu^{-3/4})\right).
        \end{equation*}
    \end{proposition}
    \begin{proof}
        We carry out Dirichlet--Neumann bracketing; see, for example, \cite{frank_schrodinger_2022}. That is, we start by further splitting the operators in \eqref{eqn:DNBrack} into equal-sized intervals scaled in $h$, applying Dirichlet and Neumann boundary conditions. Let $\ell,L>0$ and define \[I_{j,k}(\ell,L;h)=h^{1/3}(\ell j,\ell(j+1))\times h^{2/3}(Lk,L(k+1))\] with $j\in \Z$ and $k\in \N_0$. 

        Starting with the Neumann operator constructed above, we restrict it to each interval $I_{j,k}$ that intersects with $\mathcal{W}_h$ and impose Neumann boundary conditions whilst estimating potential from below. Since the potentials become purely attractive outside of $\mathcal{W}_h$, we can trivially extend to all intervals $j\in\Z, k\in\N_0$. Using the variational principle and \eqref{eqn:DNBrack} we obtain
        \begin{align*}
            &N(\Lc_h,h^{2/3}\mu)\\
            &\leq\sum_{(j,k)\in \Z\times\N_0}N\left(-h^2\Delta^N_{I_{j,k}(\ell,L;h)}+\min_{I_{j,k}(\ell,L;h)}\left(\frac{\kappa_0}{2}s^2+t\right),h^{2/3}\left(\mu +h^{1/3-5\eta}\right)\right)\\ &=\sum_{j,k}\#\Bigg\{(m,n)\in \N_0^2\colon \frac{h^2\pi^2 m^2}{h^{2/3}\ell^2}+\frac{h^2 \pi^2 n^2}{h^{4/3} L^2}+\min_{I_{j,k}(\ell,L;h)}\left(\frac{\kappa_0}{2}s^2+t\right)< h^{2/3}\left(\mu+h^{1/3-5\eta}\right)\Bigg\}\\
            &= \sum_{j,k}\sum_{n\in\N_0}\#\Bigg\{m\in \N_0\colon \frac{h^{4/3}\pi^2 m^2}{\ell^2}<h^{2/3}\Bigg(\mu-\frac{\pi^2 n^2}{L^2}-\min_{I_{j,k}(\ell,L;1)}\left(\frac{\kappa_0}{2}s^2+t\right)+h^{1/3-5\eta} \Bigg)_+\Bigg\}.
        \end{align*}
        Since each of the sums in $j,k$ and $n$ are finite and independent of $h$, after discounting the zero terms, we see that
        \begin{align*}
            \limsup_{h\rightarrow 0_+}h^{1/3}N(\Lc_h,h^{2/3}\mu)\leq&\frac{1}{\pi}\sum_{j,k,n}\ell \left(\mu-\frac{\pi^2 n^2}{L^2}-\min_{I_{j,k}(\ell,L;1)}\left(\frac{\kappa_0}{2} s^2+t\right)\right)_+^{1/2}.
        \end{align*}
        Integrating with respect to $n, j$ and $k$ appropriately we find that
        \begin{align*}
            &\limsup_{h\rightarrow 0_+}h^{1/3}N(\Lc_h,h^{2/3}\mu)\\
            &\leq\frac{1}{\pi}\sum_{j,k}\ell\Bigg[\left(\mu-\min_{I_{j,k}(\ell,L;1)}\left(\frac{\kappa_0}{2}s^2+t\right)\right)_{+}^{1/2}+\frac{L}{4}\left(\mu-\min_{I_{j,k}(\ell,L;1)}\left(\frac{\kappa_0}{2}s^2+t\right)\right)_+ \Bigg]\\
            &\leq \frac{4}{\pi\sqrt{2\kappa_0}}\Bigg[ \left( \frac{ \pi}{8 L}\mu^2+\frac{1}{15}\mu ^{5/2}\right)+\frac{\pi}{4}\mu+\frac{L \mu^{3/2} }{6}\Bigg]+\frac{2\ell} {\pi}\Bigg[ \left( \frac{2}{3L}\mu^{3/2}+\frac{1}{8}\mu^{2}\right)+\sqrt{\mu}+\frac{L\mu}{4}\Bigg].
        \end{align*}
        Hence by taking $\ell \rightarrow 0_+$ we arrive at the following:
        \begin{align}\label{eqn:sec1upasym}
             \limsup_{h\rightarrow 0_+}h^{1/3}N(\Lc_h,h^{2/3}\mu)\leq \frac{4}{15\pi\sqrt{2\kappa_0}}\mu ^{5/2}+\frac{2L \mu^{3/2} }{3\pi\sqrt{2\kappa_0}}+\frac{\pi}{2L\pi\sqrt{2\kappa_0}}\mu^2+\frac{1}{\sqrt{2\kappa_0}}\mu.
        \end{align}

        For the lower bound, we restrict ourselves to the intervals contained in $\mathcal{W}_h$ and impose Dirichlet boundary conditions. It follows from bracketing and \eqref{eqn:DNBrack} that
        \begin{align*}
            &N(\Lc_h,h^{2/3}\mu)\\
            &\geq\sum_{(j,k)\colon I_{j,k}(\ell,L;h)\subset \mathcal{W}_h}N\left(-h^2\Delta^D_{I_{j,k}(\ell,L;h)}+\min_{I_{j,k}(\ell,L;h)}\left(\frac{\kappa_0}{2}s^2+t\right),h^{2/3}\left(\mu -h^{1/3-5\eta}\right)\right)\\ 
            &= \sum_{j,k}\#\Bigg \{(m,n)\in \N^2\colon \frac{h^2\pi^2 m^2}{h^{2/3}\ell^2}+\frac{h^2 \pi^2 n^2}{h^{4/3}L^2}+\max_{I_{j,k}(\ell,L;h)}\left(\frac{\kappa_0}{2}s^2+t\right)< h^{2/3}\left(\mu-h^{1/3-5\eta}\right)\Bigg \}\\
            &= \sum_{j,k}\sum_{n\in \N}\#\Bigg\{m\in \N\colon \frac{h^{4/3}\pi^2 m^2}{\ell^2}<h^{2/3}\Bigg(\mu-h^{1/3-5\eta}-\frac{\pi^2 n^2}{L^2}-\max_{I_{j,k(\ell,L;1)}}\left(\frac{\kappa_0}{2}s^2+t\right)\Bigg)_+\Bigg\}.
        \end{align*}
        Similar to before, we use the finiteness of the sums and integrate in $n, j$ and $k$ to get 
        \begin{align*}
            &\liminf_{h\rightarrow 0_+}h^{1/3}N(\Lc_h,h^{2/3}\mu)\\
            &\geq\frac{1}{\pi}\sum_{j,k,n}\ell \left(\mu-\frac{\pi^2 n^2}{L^2}-\max_{I_{j,k}(\ell,L;1)}\left(\frac{\kappa_0}{2}s^2+t\right)\right)_+^{1/2}\\
            &\geq \frac{1}{\pi}\sum_{j,k}\ell \Bigg[\frac{L}{4}\left(\mu-\max_{I_{j,k}(\ell,L;1)}\left(\frac{\kappa_0}{2}s^2+t\right)\right)_+-\left(\mu-\max_{I_{j,k}(\ell,L;1)}\left(\frac{\kappa_0}{2}s^2+t\right)\right)_+^{1/2}\Bigg]\\
            &\geq  \frac{4}{\pi\sqrt{2\kappa_0}}\Bigg[\frac{\mu^{5/2}}{15}\left(1-\frac{L}{\mu}\right)^{5/2}-\frac{\pi \mu^2}{8L}\Bigg]-\frac{\ell \mu^{2}}{4\pi}\left(1-\frac{L}{\mu}\right)^2.
        \end{align*}
        Now letting $\ell\rightarrow 0_+$, provided that we take $L\leq \mu$, we obtain 
        \begin{equation}\label{eqn:sec1lowasy}
            \liminf_{h\downarrow 0}h^{1/3}N(\mathcal{L}_h,h^{2/3}\mu)\geq \frac{4\mu^{5/2}}{15\pi\sqrt{2\kappa_0}}-\frac{2L\mu^{3/2}}{3\pi\sqrt{2\kappa_0}}-\frac{\pi \mu^2}{2L\pi\sqrt{2\kappa_0}}.
        \end{equation}
        To see the asymptotics as $\mu\rightarrow \infty$, we combine the limits \eqref{eqn:sec1upasym} and \eqref{eqn:sec1lowasy} choosing $L$ proportional to $\mu$. We see that choosing $L=\mu^{1/4}$ leads to the best order in the remainder term. 
    \end{proof}
    We note that the asymptotics we derive differ from the classical Weyl asymptotics \eqref{eqn:weylscounting} applied to the potential $x_1-\mu h^{2/3}$. However, a correct leading-order upper bound for $\mu\rightarrow \infty$ can be derived using the CLR-type bound found in \cite{laptev_calogero_2022}, which is applicable due to the monotonicity of the potential. In the following section, we demonstrate in Remark \ref{rem:simpleproof} that a simpler proof of this result is achievable through rescaling and a double application of Weyl's law.
    \section{Proof of Theorem \ref{thm:maintrace}}\label{sec:maintrace}
    In this section we prove a generalisation of Theorem \ref{thm:maintrace}. For $V\in C_0^\infty (\R^2_+)$ we define the rescaled potential $V_h$ as a function in $\Omega$, zero outside of $\mathcal{W}_h$ with
    \begin{equation}\label{eqn:rescaledV}
        V_h(\tau(s,t))=V(h^{-1/3}s,h^{-2/3}t) \text{ for }(s,t)\in \mathcal{W}_h.
    \end{equation}
    The idea is to introduce a potential that acts on the same scale as the low-lying eigenvalues we are concerned with. In this case, instead of bracketing in the orthogonal coordinate $t$ we use separation of variables and the construction of quasi-states to extract the eigenvalues of the operator 
    \begin{equation}\label{eqn:orthogop}
        L(s)=-\frac{\mathrm{d}^2}{\mathrm{d}t^2}+t+V(s,t)\text{ in }L^2(\R_+,\dd t)
    \end{equation}
    with Dirichlet boundary conditions, which we denote by $\{\lambda_k(s;V)\}_{k=1}$. We note that for smooth and compactly supported $V$, these eigenvalues are well-defined continuous functions of $s$. Before stating our result, we need a technical lemma from \cite{cornean_two-dimensional_2022}.
    \begin{lemma}\label{lem:expdecay}
        Let $\Lambda>0$ and $V\in C_0^\infty(\R_+)$ then there exist positive and finite constants $R=R\left(\Lambda,\norm{V}_\infty\right)$ and $C=C\left(\Lambda,\norm{V}_\infty\right)$ such that for any $\ell>R$ and any eigenvalue $\lambda<\Lambda$ of the operator 
        \begin{equation*}
            -\frac{\mathrm{d}^2}{\mathrm{d}t^2}+t+V(t) \text{ in } L^2(0,\ell),
        \end{equation*}
        with Dirichlet conditions at $t=0$ and Dirichlet or Neumann conditions at $t=\ell$, the corresponding eigenfunction $\phi_\lambda$ satisfies 
        \begin{align*}
            \int_{R}^\ell \left(\abs{\phi_\lambda^\prime}^2+ \abs{\phi_\lambda}^2\right) e^{t^{3/2}}\dd t&\leq C\norm{\phi_\lambda}_{L^2(0,\ell)}^2.
        \end{align*}
    \end{lemma}
    This result is essentially identical to \cite[Proposition 2.1]{cornean_two-dimensional_2022}. The only difference is that we work on a different scale and use the boundedness of $V$. 
    \begin{proposition}\label{prop:asymppotential}\label{prop:preciseh23}
        Let $\gamma\geq 0$, $\mu\geq 0$ and $V\in C_0^\infty(\R^2_+)$, then 
        \begin{equation*}\label{eqn:asymppoten}
            \lim_{h\rightarrow 0_+}h^{(1-2\gamma)/3}\Tr\left(\Lc_h+h^{2/3}V_h-h^{2/3}\mu\right)_{-}^\gamma=L_{\gamma,1}^{\mathrm{cl}}\sum_{j=1}\int_{\R}\left(\mu-\frac{\kappa_0}{2}s^2-\lambda_j(s;V)\right)^{\gamma+1/2}_+\dd s.
        \end{equation*}
    \end{proposition}
    Taking $V\equiv 0$, it follows that the eigenvalues extracted from \eqref{eqn:orthogop} are just Airy zeros, independent of $s$. Putting this into the expression above precisely yields the first result in Theorem \ref{thm:maintrace}. 
    \begin{proof}
     To ease notation, we label 
     \begin{equation*}
         \mathcal{L}_h(\mu, V)\coloneqq \Lc_h+h^{2/3}\left(V_h-\mu\right).
     \end{equation*}
     We fix $\gamma\geq 0$ and start by recalling the construction in the previous section with the addition of a potential. Due to the boundedness of the potential $V_h$ it follows from \eqref{eqn:bracketing} and \eqref{eqn:DNBrack} that for sufficiently small $h$, 
     \begin{equation}\label{eqn:boundprop1}
        \begin{split}
        \Tr\Lc_h\left(\mu,V\right)_{-}^\gamma &\geq \Tr\left(-h^2\Delta_{\mathcal{W}_h}^D+\frac{\kappa_0}{2}s^2+t+h^{2/3}\left(V_h\circ\tau-\mu+h^{1/3-5\eta}\right)\right)_{-}^\gamma \\
    \Tr\Lc_h\left(\mu,V\right)_{-}^\gamma &\leq \Tr\left(-h^2\Delta_{\mathcal{W}_h}^M+\frac{\kappa_0}{2}s^2+t+h^{2/3}\left(V_h\circ\tau-\mu-h^{1/3-5\eta}\right)\right)_{-}^\gamma
         \end{split}
     \end{equation}
     with $M$ here denoting mixed boundary conditions, where we keep Dirichlet boundary conditions where $t=0$ and impose Neumann conditions elsewhere. In deriving 
     \eqref{eqn:boundprop1} from \eqref{eqn:DNBrack} we have used that 
     \begin{equation*}
         \abs{t\kappa(s) h^{2/3} V_h(\tau(s,t))}\lesssim h^{4/3-\eta}\norm{V}_\infty.
     \end{equation*}
     
     Now we carry out a change of scale directly and use separation of variables. Applying the unitary transformation $\mathcal{U}_h\phi(s,t)=h^{-1/2}\phi(h^{-1/3}s,h^{-2/3}t)$ to the operators on the right in \eqref{eqn:boundprop1} we obtain
     \begin{align*}
        -h^{4/3}\partial_s^2+h^{2/3}\left(-\partial_t^2+t+\frac{\kappa_0}{2}s^2+V(s,t)\right) \text{ in } L^2((-h^{-\eta},h^{-\eta})\times(0,h^{-\eta})).
     \end{align*}
     with their respective boundary conditions. It is then helpful to reformulate these, writing them as one-dimensional Schr\"odinger operators in $s$ with operator-valued potentials. In this form, the operators with Dirichlet and mixed boundary conditions are, up to a factor of $h^{2/3}$, given by 
     \begin{align*}
         -&h^{2/3}\frac{\mathrm{d}^2}{\mathrm{d}s^2}\Big\vert^D_{(-h^{-\eta},h^{-\eta})}\otimes \id+\frac{\kappa_0}{2}s^2\otimes\id +\mathcal{V}^D_h (s), \text{ and }\\
         -&h^{2/3}\frac{\mathrm{d}^2}{\mathrm{d}s^2}\Big\vert^N_{(-h^{-\eta},h^{-\eta})}\otimes \id+\frac{\kappa_0}{2}s^2\otimes\id +\mathcal{V}^M_h(s)
     \end{align*}
     in $L^2((-h^{-\eta},h^{-\eta}),\dd s;L^2(0,h^{-\eta}))$, where for each $s$
        \begin{align*}
            \mathcal{V}^D_h(s)&=-\frac{\mathrm{d}^2}{\mathrm{d}t^2}\Big\vert^D_{(0,h^{-\eta})}+t+V(s,t)\text{ and }\\
            \mathcal{V}^M_h(s)&=-\frac{\mathrm{d}^2}{\mathrm{d}t^2}\Big\vert^M_{(0,h^{-\eta})}+t+V(s,t)
        \end{align*}
    as operators in $L^2(0,h^{-\eta})$, where $M$ symbolises the imposition of Dirichlet conditions at $t=0$ and Neumann conditions at $t=h^{-\eta}$. Furthermore, we can fix the domain of the operators in $s$. Noting that 
    \begin{equation*}
        \lambda_k(\mathcal{V}_h^M(s))\geq-\norm{V}_{L^\infty(\R_+^2)} 
    \end{equation*}
    we can restrict the Neumann operator to the interval $(-\widetilde{R},\widetilde{R})$ with 
    \begin{equation}\label{eqn:R}
    \widetilde{R}=\sqrt{2/\kappa_0}(\norm{V}_\infty+2\mu)^{1/2}
    \end{equation}
    so that the potential is purely repulsive outside of this set. Whilst for the Dirichlet case we can restrict it to any smaller interval and use domain monotonicity.

    Thus, we obtain from \eqref{eqn:boundprop1} and the above that for every $\varepsilon\in (1/2,1)$ and $R>0$ there exists $h^\prime>0$ such that for all $h<h^\prime$,
        \begin{equation}\label{eqn:prop1finalbound}
            \begin{split}
            \Tr\Lc_h\left(\mu,V\right)_{-}^\gamma\geq &h^{2\gamma/3}\sum_{k=1}\Tr\left(-h^{2/3}\frac{\mathrm{d}^2}{\mathrm{d}s^2}\Big\vert^D_{(-R,R)}+\frac{\kappa_0}{2}s^2+\lambda_k\left(\mathcal{V}^D_h(s)\right)-\varepsilon\mu\right)^\gamma_{-}\\
            \Tr\Lc_h\left(\mu,V\right)_{-}^\gamma \leq &h^{2\gamma/3}\sum_{k=1}\Tr\left(-h^{2/3}\frac{\mathrm{d}^2}{\mathrm{d}s^2}\Big\vert^N_{(-\widetilde{R},\widetilde{R})}+\frac{\kappa_0}{2}s^2+\lambda_k\left(\mathcal{V}^M_h(s)\right)-\varepsilon^{-1}\mu \right)^\gamma_{-}.
            \end{split}
        \end{equation}
        Moreover, we note that the number of eigenvalues of $\mathcal{V}_h^M(s)$ and $\mathcal{V}_h^D(s)$ that we need to consider in the sums above are finite, uniformly in $s\in\R$ and $h<h^\prime$, in particular
        \begin{equation*}
            \widetilde{N}\coloneqq \max_{\varepsilon\in (1/2,1)}\big\{N(\mathcal{V}^D_h(s),\varepsilon\mu),N(\mathcal{V}^M_h(s),\varepsilon^{-1}\mu)\big\}\lesssim_{\mu} \norm{V}_{L^\infty(\R^2_+)} .
        \end{equation*}
        The idea now is to use Lemma \ref{lem:expdecay} to show that the eigenvalues $\lambda_k(\mathcal{V}^D(s))$ and $\lambda_k(\mathcal{V}^M(s))$ converge to the eigenvalues of the Dirichlet operator $L(s)$ as $h\rightarrow 0_+$, uniformly in $s$. Then we can apply the standard form of Weyl's law for $\gamma-$Riesz means of Schr\"odinger operators on finite intervals with Dirichlet and Neumann conditions to each of the operators above; see, for example, \cite{frank_weyls_2023}. Finally, by taking $R\rightarrow \infty$ and $\varepsilon\rightarrow 1_{-}$ we will obtain the result.

        We begin with the eigenvalues of $\mathcal{V}^M_h(s)$. Fixing $s\in \R$, we denote by $\phi_{k,h}$ the eigenfunction corresponding to $\lambda_k(\mathcal{V}_h^M(s))$, satisfying Neumann conditions at $t=h^{-\eta}$. Take $\chi\in C^\infty(\R)$ with $0\leq \chi\leq 1$ with $\chi(t)=1$ for $t<0$ and $\chi(t)=0$ for $t\geq 1$ and such that $\norm{\partial_t\chi}_\infty < \infty$. Then for $0<\widetilde{\eta}<\eta$ define $\chi_h(t)=\chi(t-h^{-\widetilde{\eta}})$. It follows that the cut-off functions $\chi_h\phi_{k,h}$ lie in the form domain of the operator $L(s)$ given by \eqref{eqn:orthogop}, after being trivially extended by zero. The min--max principle for eigenvalues then yields
        \begin{align*}
            \lambda_1(s;V)\leq \frac{\left(L(s)\phi_{1,h}\chi_h,\phi_{1,h}\chi_h\right)_{L^2(\R_+)}}{\norm{\phi_{1,h}\chi_h}_{L^2(\R_+)}^2}&=\frac{\left(\mathcal{V}^M_{h}(s)\phi_{1,h}\chi_h,\phi_{1,h}\chi_h\right)_{L^2(0,h^{-\eta})}}{\norm{\phi_{1,h}\chi_h}_{L^2(0,h^{-\eta})}^2}\\
            &\leq \lambda_1(\mathcal{V}^M_h(s))+C e^{- h^{-3\widetilde{\eta}/2}}
        \end{align*}
        with a finite constant $C<\infty$ that is independent of $h$ and $s$. In the last line we have used the decay estimate from Lemma \ref{lem:expdecay} together with the boundedness of $\chi$ and $\partial_t\chi$, to show that for all $1\leq k\leq \widetilde{N}$ we have 
        \begin{align*}
            \abs{(\mathcal{V}^M_{h}\phi_{k,h},\phi_{k,h})_{L^2(0,h^{-\eta})}-\left(\mathcal{V}^M_{h}\phi_{k,h}\chi_h,\phi_{k,h}\chi_h\right)_{L^2(0,h^{-\eta})}}&\lesssim e^{- h^{-3\widetilde{\eta}/2}}\norm{\phi_{k,h}}_{L^2(0,h^{-\eta})}^2\text{, and }\\
            \abs{\norm{\chi_h\phi_{k,h}}_{L^2(0,h^{-\eta})}^2-\norm{\phi_{k,h}}_{L^2(0,h^{-\eta})}^2}&\lesssim e^{- h^{-3\widetilde{\eta}/2}}\norm{\phi_{k,h}}_{L^2(0,h^{-\eta})}^2
        \end{align*}
        where the implicit constants depend only on $\chi$, $\mu$ and $\norm{V}_{L^\infty(\R^2_+)}$.

        To deduce a similar statement for the higher eigenvalues, we note that
        \begin{equation*}
            \abs{\left(\chi \phi_{k,h},\chi_h \phi_{j,h}\right)_{L^2(\R_+)}-\delta_{jk}}\lesssim e^{- h^{-3\eta/2}},
        \end{equation*}
        uniformly in $1\leq j,k\leq \widetilde{N}$ and independent of $s$. Thus, $h$ can be chosen sufficiently small so that for all $k\leq \widetilde{N}$ the set $\{\chi_h \phi_{j,h}\}_{k=1}^k$ forms a $k$ dimensional subspace of $L^2(\R_+)$. Therefore, by the min-max principle and the decay estimates above, we have  
        \begin{equation}\label{eqn:minmaxMixed}
            \begin{split}
            \lambda_k(s;V)\leq \max_{\phi\in \{\chi_h \phi^h_k\}_{j=1}^k}\frac{(L(s)\phi,\phi)_{L^2(\R_+)}}{\norm{\phi}_{L^2(\R_+)}}&=\frac{\left(\mathcal{V}^M_{h}(s)\phi_{k,h}\chi_h,\phi_{k,h}\chi_h\right)_{L^2(0,h^{-\eta})}}{\norm{\phi_{k,h}\chi_h}_{L^2(0,h^{-\eta})}^2}\\&\leq \lambda_k(\mathcal{V}_h^M(s))+\widetilde{C} e^{-h^{-3\widetilde{\eta}/2}}
            \end{split}
        \end{equation}
        for all $k\leq \widetilde{N}$, with the constant $\widetilde{C}<\infty$ independent of $k$, $h$ and $s$. 

        We now turn to the Dirichlet operator $\mathcal{V}_h^D(s)$. This time we fix $s\in \R$ and work with the eigenfunctions of $L(s)$, which we denote by $\psi_k$. Cutting these off in the set $(0,h^{-\eta})$ and using estimates analogous to the above from Lemma \ref{lem:expdecay} we see that there exists $h$ sufficiently small so that for all $k\leq \widetilde{N}$
        \begin{equation}\label{eqn:minmaxDirich}
            \begin{split}\lambda_k(\mathcal{V}_h^D(s))\leq\max_{\psi\in\{\psi_j\}_{j=1}^k} \frac{(\mathcal{V}_h^D(s) \psi_1,\psi_1)_{L^2(0,h^{-\eta})}}{\norm{\psi_1}^2_{L^2(0,h^{-\eta})}}&=\frac{(L(s) \psi_k,\psi_k)_{L^2(0,h^{-\eta})}}{\norm{\psi_k}^2_{L^2(0,h^{-\eta})}}\\
            &\leq \lambda_k(s;V)+\widetilde{\widetilde{C}} e^{-h^{-3\eta/2}}
            \end{split}
        \end{equation}
        with some constant $\widetilde{\widetilde{C}}<\infty$ independent of $k$, $h$ and $s$. 

        We insert the estimates \eqref{eqn:minmaxMixed} and \eqref{eqn:minmaxDirich} in \eqref{eqn:prop1finalbound}, where we incorporate the errors into $\varepsilon$, noting that we can still take it as close to $1$ for all $h<h^{\prime}$ with $h^\prime$ small. Then after applying Weyl's law, it follows that 
        \begin{equation*}
            \limsup_{h\rightarrow 0_+}h^{(1-2\gamma)/3}\Tr\Lc_h\left(\mu,V\right)_{-}^\gamma\leq L_{\gamma,1}^\mathrm{cl}\sum_{k=1}\int_{\R}\left(\varepsilon^{-1} \mu-\lambda_k(s,V)-\frac{\kappa_0}{2}s^2\right)_+^{\gamma+1/2}\dd s
        \end{equation*}
        and
        \begin{align*}
            \liminf_{h\rightarrow 0_+}h^{(1-2\gamma)/3}\Tr\Lc_h\left(\mu,V\right)_{-}^\gamma &\geq L_{\gamma,1}^{\mathrm{cl}}\sum_{j}\int_{-R}^R \left(\varepsilon\mu-\lambda_{k}(s;V)-\frac{\kappa_0}{2}s^2\right)_+^{\gamma+1/2}\dd s,
        \end{align*} 
        thus by taking $\varepsilon\rightarrow 1_+$ and $R\rightarrow \infty$ we obtain the result. 
    \end{proof}
    
    \begin{remark}
        The assumption in Proposition \ref{prop:preciseh23} that $V\in C_0^\infty(\R_+^2)$ can be relaxed. It is possible to extend to the class of $V\in L^{\gamma+1}(\R^2_+)$ using an approximation argument similar to that in \cite[Section 4.7]{frank_schrodinger_2022}. 
        
        Take a sequence $C_0^\infty(\R^2_+)\ni V^{(n)}\rightarrow V$ and denote the rescaling of the former according to \eqref{eqn:rescaledV} by $V_h^{(n)}$. Then for any $\theta\in (0,1)$ one can split our operator into 
        \begin{align*}
           \Lc_h(\mu,V)=&\left(-h^2(1-\theta)\Delta_\Omega^D+x_1+h^{2/3}(V^{(n)}_{h}-\mu)\right)+\left(-h^2\theta\Delta_\Omega^D+h^{2/3}(V_h-V^{(n)}_{h})\right).
        \end{align*}
        and see that $N(\Lc_h(\mu,V),0)$ is bounded from above by
        \begin{equation*}
            N\left(-h^2(1-\theta)\Delta_\Omega^D+x_1+h^{2/3}V^{(n)}_{h}, h^{2/3}\mu\right)+N\left(-h^2\theta\Delta_\Omega^D+h^{2/3}(V_h-V^{(n)}_{h}),0\right). 
        \end{equation*}
        For the first term, we can apply Proposition \ref{prop:preciseh23} and for the second term, use the CLR type bound found in \cite{frank_bound_2018} to control the limit as $h\rightarrow 0_+$. Taking $n\rightarrow \infty$ and $\theta\rightarrow 0_+$ yields the correct upper bound. The argument for the lower bound and for the $\gamma-$Riesz means is similar.
    \end{remark}

    \begin{remark}\label{rem:simpleproof}
        The above analysis leads to a simpler proof of Proposition \ref{prop:roughweyl}. To see this, note that if we don't approximate using quasi-states then Weyl's law applied to the statement \eqref{eqn:prop1finalbound} yields that for any $R>0$
        \begin{align*}
            \liminf_{h\rightarrow 0_+}h^{1/3}N\left(\Lc_h,\mu h^{2/3}\right)&\geq \int_{\R}\sum_{k=1}\left(\mu-\frac{\kappa_0}{2}s^2-\lambda_k(\mathcal{V}^D_R)\right)^{1/2}_+\dd s\\
            &=\int_{\R}\Tr\left(\mathcal{V}^D_R+\frac{\kappa_0}{2}s^2-\mu\right)^{1/2}_{-}\dd s\\
            &=\int_{\R}\Tr\left(\frac{\mathrm{d}^2}{\mathrm{d}t^2}\Big\vert^D_{(0,R^{-\eta})}+t+\frac{\kappa_0}{2}s^2-\mu\right)^{1/2}_{-}\dd s.
        \end{align*}
        Applying Weyl's law again for the trace of the operator in $t$ as $\mu\rightarrow \infty$ one obtains a bound from below with the same leading-order term as Proposition \ref{prop:roughweyl}. The same argument applied to the Neumann operator gives the result. 
    \end{remark}
    Next, look at the asymptotic number of eigenvalues between the second and third levels in \eqref{eqn:eigencornean}. Let $\alpha\in (2/3,1)$, then we consider for $V\in C_0^\infty(\R_+^2)$ a modified form of rescaled potential $V_{h,\alpha}$, supported in $\tau(\mathcal{W}_h)$, with \begin{equation}\label{eqn:rescalalpha}
    V_{h,\alpha}(\tau(s,t))=V(h^{-\alpha/2}s,h^{-2/3}t).
    \end{equation}
    To simplify the notation, we employ
     \begin{equation*}
         \mathcal{L}_h(\mu, V;\alpha)\coloneqq \Lc_h-z_1 h^{2/3}+h^\alpha\left(V_{h,\alpha}-\mu\right)
     \end{equation*}
     and find asymptotics for the sums of its negative eigenvalues. The crucial element here is that we find an explicit dependence on the normalised Airy function $\mathrm{a}_1$ given by \eqref{eqn:airnorm}. That is, the eigenfunction of the operator $L(s)$ given in \eqref{eqn:orthogop}, which arises in the following result from linear perturbation theory. 
    \begin{proposition}\label{prop:precisealpha}
        Let $\gamma\geq 0$, $\alpha\in (2/3,1)$, $\mu\geq 0$ and $V\in C_0^\infty(\R^2_+)$, then 
        \begin{equation*}
            \lim_{h\rightarrow 0_+} h^{1-\alpha(1+\gamma)} \Tr \Lc\left(\mu,V;\alpha\right)_{-}^\gamma=L_{\gamma,1}^{\mathrm{cl}}\int_{\R}\left(\mu-\frac{\kappa_0}{2}s^2-\int_{\R_+}V(s,t)\mathrm{a}_1(t)^2\dd t\right)^{\gamma+1/2}_+\dd s.
        \end{equation*}
    \end{proposition}
    \begin{proof}
     We begin by fixing $\gamma\geq 0$ and choosing $\eta\in (0,(1-\alpha)/5)$. The latter ensures that the errors introduced in \eqref{eqn:DNBrack} can be kept on a scale of $o_{h\rightarrow 0_+}(h^{\alpha})$. Then we find that for any $\varepsilon\in (1/2,1)$ there exists $h^\prime$ such that for all $h<h^\prime$ 
     \begin{equation*}
        \begin{split}
        \Tr\Lc_h(\mu,V;\alpha)_{-}^\gamma &\geq \Tr\left(-h^2\Delta_{\mathcal{W}_h}^D+\frac{\kappa_0}{2}s^2+t-z_1h^{2/3}+h^{\alpha}\left(V_{h,\alpha}\circ\tau-\varepsilon\mu\right)\right)_{-}^\gamma \\
        \Tr\Lc_h(\mu,V;\alpha)_{-}^\gamma &\leq \Tr\left(-h^2\Delta_{\mathcal{W}_h}^M+\frac{\kappa_0}{2}s^2+t-z_1h^{2/3}+h^{\alpha}\left(V_{h,\alpha}\circ\tau-\varepsilon^{-1}\mu\right)\right)_{-}^\gamma
         \end{split}
     \end{equation*}
     where we have used the boundedness of $V$ in $\R^2_+$.
     
     Applying a change of scale induced by the unitary transformation $\mathcal{U}_{h,\alpha}\phi(s,t)=h^{-1/3-\alpha/4}\phi(h^{-\alpha/2}s,h^{-2/3}t)$ to the operators above we obtain
     \begin{align*}
        -h^{2-\alpha}\partial_s^2+h^{\alpha}\frac{\kappa_0}{2}s^2+h^{2/3}\left(-\partial_t^2+t+h^{\alpha-2/3}V(s,t)\right)
     \end{align*}
     in the rescaled domain, with their respective boundary conditions. We then think of these operators as one-dimensional Schr\"odinger operators in $s$ with operator-valued potentials. We find that operators with Dirichlet and mixed boundary conditions are, up to a factor of $h^{2/3}$, given by 
     \begin{align*}
         -&h^{4/3-\alpha}\frac{\mathrm{d}^2}{\mathrm{d}s^2}\Big\vert^D_{(-h^{1/3-\alpha/2-\eta},h^{1/3-\alpha/2-\eta})}\otimes \id+h^{\alpha-2/3}\frac{\kappa_0}{2}s^2\otimes\id +\mathcal{V} D_h (s), \text{ and }\\
         -&h^{4/3-\alpha}\frac{\mathrm{d}^2}{\mathrm{d}s^2}\Big\vert^N_{(-h^{1/3-\alpha/2-\eta},h^{1/3-\alpha/2-\eta})}\otimes \id+h^{\alpha-2/3}\frac{\kappa_0}{2}s^2\otimes\id +\mathcal{V}^M_h(s)
     \end{align*}
     in $L^2((-h^{1/3-\alpha/2-\eta},h^{1/3-\alpha/2-\eta}),\dd s;L^2(0,h^{-\eta}))$, where for each $s$
        \begin{align*}
            \mathcal{V}^D_h(s;\alpha)&=-\frac{\mathrm{d}^2}{\mathrm{d}t^2}\Big\vert^D_{(0,h^{-\eta})}+t+h^{\alpha-2/3}V(s,t)\text{ and }\\
            \mathcal{V}^M_h(s;\alpha)&=-\frac{\mathrm{d}^2}{\mathrm{d}t^2}\Big\vert^M_{(0,h^{-\eta})}+t+h^{\alpha-2/3}V(s,t)
        \end{align*}
    as operators in $L^2(0,h^{-\eta})$, with $M$ denoting mixed conditions as before.

        Therefore, together with domain monotonicity we see that for every $\varepsilon\in (1/2,1)$ and $R>0$ there exists $h^\prime >0$ such that for all $h<h^\prime$, $\Tr\Lc_h\left(\mu,V;\alpha\right)_{-}^\gamma$ is bounded from below by 
        \begin{align}
            &h^{2\gamma/3}\sum_{k=1}\Tr\left(-h^{4/3-\alpha}\frac{\mathrm{d}^2}{\mathrm{d}s^2}\Big\vert^D_{(-R,R)}+h^{\alpha-2/3}\left(\frac{\kappa_0}{2}s^2-\mu\right)+\lambda_k\left(\mathcal{V}^D_h(s;\alpha)\right)-z_1\right)^\gamma_{-}\label{eqn:firstalphaprop}\intertext{and bounded from above by}&h^{2\gamma/3}\sum_{k=1}\Tr\left(-h^{4/3-\alpha}\frac{\mathrm{d}^2}{\mathrm{d}s^2}\Big\vert^N_{(-\widetilde{R},\widetilde{R})}+h^{\alpha-2/3}\left(\frac{\kappa_0}{2}s^2-\mu\right)+\lambda_k\left(\mathcal{V}^M_h(s;\alpha)\right)-z_1 \right)^\gamma_{-}\label{eqn:secondalphaprop}
        \end{align}
        where $\widetilde{R}$ is given by \eqref{eqn:R}. The approach now is to use Lemma \ref{lem:expdecay} to show that the first eigenvalues of $\mathcal{V}^D_h(s;\alpha)$ and $\mathcal{V}^M_h(s;\alpha)$ converge to the first eigenvalue of the operator \[L_{h}(s;\alpha)\coloneqq-\frac{\mathrm{d}^2}{\mathrm{d}t^2}+t+h^{\alpha-2/3}V(s,t)\] in $L^2(\R_+)$ with Dirichlet conditions, as $h\rightarrow 0_+$, uniformly in $s$. To do this, we use regular perturbation theory for the eigenvalues of $L_h(s;\alpha)$, see, for example, \cite[Section XII.2]{reed_methods4}. It follows that for any fixed $k\geq 1$, if $\phi_k$ is the normalised $k$th eigenfunction of $L_0(s;\alpha)$, then
        \begin{equation}
            \begin{split}\label{eqn:perturbalpha}
            \lambda_k(L_h(s;\alpha))&=\lambda_k(L_0(s;\alpha))+h^{\alpha-2/3}\int_{\R_+}V(s,t)\phi_k(t)^2 \dd t+\mathcal{O}_{h\rightarrow 0_+}\left(h^{2\alpha-4/3}\right)\\
            &=z_k+h^{\alpha-2/3}\int_{\R_+}V(s,t)\mathrm{a}_k(t)^2\dd t+\mathcal{O}_{h\rightarrow 0_+}\left(h^{2\alpha-4/3}\right).
            \end{split}
        \end{equation}
        The fact that for any given $k$ the error term in \eqref{eqn:perturbalpha} is uniformly finite in $s$ can be seen from the boundedness and compact support of $V$. It is then clear from \eqref{eqn:perturbalpha} that we only need to consider the first eigenvalue in both \eqref{eqn:firstalphaprop} and \eqref{eqn:secondalphaprop}, since all other terms will be zero for suitably small $h$.

        Then we perform the same cutting off of the eigenfunctions of $\mathcal{V}_h^M(s;\alpha)$ and of $L_h(s;\alpha)$ as in the proof of Proposition \ref{prop:preciseh23}. With the exponential decay estimate from Lemma \ref{lem:expdecay} and the min--max principle we obtain that 
        \begin{equation}
            \begin{split}\lambda_1(\mathcal{V}^M_h(s;\alpha))&\geq z_1+h^{\alpha-2/3}\int_{\R_+}\mathrm{a}_1(t)^2 V(s,t)\dd t - C h^{2\alpha-4/3}\\ \label{eqn:firsteigasymp}
            \lambda_1(\mathcal{V}^D_h(s;\alpha))
            &\leq z_1+h^{\alpha-2/3}\int_{\R_+}\mathrm{a}_1(t)^2 V(s,t)\dd t+C h^{2\alpha-4/3}
            \end{split}
        \end{equation}
        with a finite constant $C<\infty$ that is independent of $h$ and $s$.
        
        Inserting \eqref{eqn:firsteigasymp} into \eqref{eqn:firstalphaprop} and \eqref{eqn:secondalphaprop}, absorbing the error into $\varepsilon$, we conclude that
        \begin{equation*}
            \begin{split}
            \Tr\Lc_h(\mu,V;\alpha)_{-}^\gamma\geq &h^{\alpha\gamma}\Tr\left(-h^{2-2\alpha}\frac{\mathrm{d}^2}{\mathrm{d}s^2}\Big\vert^D_{(-R,R)}+\frac{\kappa_0}{2}s^2-\int_{\R_+}V(s,t)\mathrm{a}_1(t)^2\dd t-\varepsilon\mu\right)^\gamma_{-}\\
            \Tr\Lc_h(\mu,V;\alpha)_{-}^\gamma \leq &h^{\alpha\gamma}\Tr\left(-h^{2-2\alpha}\frac{\mathrm{d}^2}{\mathrm{d}s^2}\Big\vert^N_{(-\widetilde{R},\widetilde{R})}+\frac{\kappa_0}{2}s^2+\int_{\R_+}V(s,t)\mathrm{a}_1(t)^2\dd t-\varepsilon^{-1}\mu \right)^\gamma_{-}.
            \end{split}
        \end{equation*}
        Then by applying Weyl's law to these operators and taking $\varepsilon\rightarrow 1_+$ and $R\rightarrow \infty$ we obtain the result. 
    \end{proof}
    Applying Proposition \ref{prop:precisealpha} to $V\equiv 0$ completes the proof of Theorem \ref{thm:maintrace}.
    
    \section{Proof of Theorem \ref{thm:maindensity}}\label{sec:proj}
    In this final section, we culminate our results from previous sections to deduce asymptotics for the density of the spectral projector onto low--lying states. The argument follows from that used by Frank in \cite{frank_weyls_2023} and uses regular perturbation theory as in Proposition \ref{prop:precisealpha}. 
    
    \begin{proof}[Proof of Theorem \ref{thm:maindensity}]
        Let $\rho_h$ denote the density of $\Gamma_h=(\mathcal{L}_h-h^{2/3}\mu)^0_{-}$. Then we fix $V\in C_0^\infty(\R_+^2)$ and take $V_h$ as the rescaling of $V$ according to \eqref{eqn:rescaledV}. It follows from the variational principle that
        \begin{align*}
            h^{2/3}\int_{\Omega}V_h(x)\rho_h(x)\dd x&=\Tr\left(\mathcal{L}_h+h^{2/3}V_h-h^{2/3}\mu\right)\Gamma_h-\Tr\left(\mathcal{L}_h-h^{2/3}\mu\right)\Gamma_h\\
            &\geq -\Tr\left(\mathcal{L}_h+h^{2/3}V_h-h^{2/3}\mu\right)_{-}+
        \Tr\left(\mathcal{L}_h-h^{2/3}\mu\right)_{-}
        \end{align*}
        where we have used equality in the second term. Thus, from Proposition \ref{prop:asymppotential} we have
        \begin{align*}
            \liminf_{h\rightarrow 0_+}h^{1/3}\int_{\Omega}V_h\rho_h\geq  \frac{2}{3\pi}\sum_{k=1}\int_{\R}\left(\mu-\frac{\kappa_0}{2}s^2-z_k\right)^{3/2}_{+}-\left(\mu-\frac{\kappa_0}{2}s^2-\lambda_k(s;V)\right)^{3/2}_{+}\dd x,
        \end{align*}
        and applying it again after replacing $V$ by $-V$ we see that
        \begin{align*}
            \limsup_{h\rightarrow 0_+}h^{1/3}\int_{\Omega}V_h\rho_h\leq \frac{2}{3\pi}\sum_{k=1}\int_{\R}\left(\mu-\frac{\kappa_0}{2}s^2-\lambda_k(s;-V)\right)^{3/2}_{+}-\left(\mu-\frac{\kappa_0}{2}s^2-z_k\right)^{3/2}_{+}\dd x.
        \end{align*}
        
        Then, considering $\varepsilon V$ instead, it follows from the above that the $\liminf$ term is bounded from below by 
        \begin{align*}
            \limsup_{\varepsilon\rightarrow 0}L_{1,1}^{\mathrm{cl}}\sum_{k=1}\int_{\R}\frac{1}{\varepsilon}\left[\left(\mu-\frac{\kappa_0}{2}s^2-z_k\right)^{3/2}_{+}-\left(\mu-\frac{\kappa_0}{2}s^2-\lambda_k(s;\varepsilon V)\right)^{3/2}_{+}\right]\dd x.
        \end{align*}
        Given that the sum in $k$ is finite, we can use  
        \begin{equation*}
            \frac{\mathrm{d}}{\mathrm{d}\varepsilon}\left(\mu-\frac{\kappa_0}{2}s^2-\lambda_k(s;\varepsilon V)\right)^{3/2}_+\Big\vert_{\varepsilon=0}=\frac{3}{2}\left(\mu-\frac{\kappa_0}{2}s^2-z_k\right)^{1/2}_+ \frac{\mathrm{d}}{\mathrm{d}\varepsilon}\lambda_k(s;\varepsilon V)\Big\vert_{\varepsilon=0},
        \end{equation*}
        where from perturbation theory it follows that 
        \begin{equation*}
            \frac{\mathrm{d}}{\mathrm{d}\varepsilon}\lambda_k(s;\varepsilon V)\Big\vert_{\varepsilon=0}=\int_{0}^\infty V(s,t)\mathrm{a}_k(t)^2\dd t
        \end{equation*}
        uniformly in $s$, which follows from the boundedness and compact support of $V$. Performing the same calculation for the $\limsup$ yields
        \begin{align}\label{eqn:finaldensityasymp}
            \lim_{h\rightarrow 0_+}h^{1/3}\int_\Omega V_h \rho_h\dd x=\frac{1}{\pi}\sum_{k=1}\int_{\R^2_+}\left(\mu-\frac{\kappa_0}{2}s^2-z_k\right)^{1/2}_{+}\mathrm{a}_k(t)^2 V(s,t)\dd s\dd t.
        \end{align}
        Now we note that the integral on the the left-hand side can be written as
        \begin{align*}
            \int_{\Omega} V_h\rho_h
            &=\int_{\mathcal{W}_h}V_h(\tau(s,t))\rho_h(\tau(s,t))m(s,t)\dd s\dd t\\
            &=h\int_{-h^{-\eta}}^{h^{-\eta}}\int_0^{h^{-\eta}}V(s,t)\rho_h(\tau(h^{1/3}s,h^{2/3}t))(1-h^{2/3}t\kappa(h^{1/3}s))\dd s\dd t.
        \end{align*}
        Thus, using the boundedness of the curvature $\kappa$ and combining this with \eqref{eqn:finaldensityasymp} we conclude that 
        \begin{align*}
            \lim_{h\rightarrow 0_+}h^{4/3}\int_{\R^2_+}V(s,t)\widetilde{\rho_h}(s,t)\dd s\dd t=\frac{1}{\pi}\int_{\R^2_+}\sum_{k=1}\left(\mu-\frac{\kappa_0}{2}s^2-z_k\right)^{1/2}_{+}\mathrm{a}_k(t)^2 V(s,t)\dd s\dd t
        \end{align*}
        where $\widetilde{\rho_h}(s,t)\coloneqq \rho_h(h^{1/3}s,h^{2/3}t)$. Since this holds for any $V\in C_0^\infty(\R_+^2)$, we obtain the first statement in Theorem \ref{thm:maindensity}. The proof of the second part follows by the same argument using the rescaled potential \eqref{eqn:rescalalpha} and Proposition \ref{prop:precisealpha}. 
    \end{proof}

    \subsection*{Acknowledgements}
    This work was funded by the Deutsche Forschungsgemeinschaft (DFG) project TRR 352 – Project-ID 470903074. The author is grateful to Rupert L. Frank for his direction and insight.

\end{document}